\documentclass[12pt]{amsart}
\usepackage{amsmath,amsfonts,amssymb}
\numberwithin{equation}{section}

\newtheorem{theorem}{Theorem}[section]
\newtheorem{lemma}[theorem]{Lemma}

\newtheorem{proposition}[theorem]{Proposition}

\DeclareMathOperator{\Det}{Det} \DeclareMathOperator{\Tr}{Tr}

 \DeclareMathOperator{\Rea}{Re}
 
\DeclareMathOperator{\Diag}{Diag}

\title [Complex Weyl correspondence...]{Complex Weyl correspondence and harmonic representation of $SU(p,q)$}
\author{Benjamin Cahen}
\address{Universit\'e de Lorraine, Site de Metz, UFR-MIM,
D\'epartement de math\'ematiques,
B\^atiment A,
3 rue Augustin Fresnel, BP 45112,
57073 METZ Cedex 03, France.}
\email{benjamin.cahen@univ-lorraine.fr}

\subjclass[2020]{22E46; 47B32; 46E22.} \keywords{
Heisenberg group; complex Weyl correspondence; unitary group;
unitary representation; harmonic representation;
reproducing kernel Hilbert space.  }

\begin{document}

\maketitle

\begin{abstract}
We study the harmonic representation of $SU(p,q)$ in connection to the complex Weyl correspondence on the Fock space. In particular, we give explicit formulas for the complex 
Weyl symbols
of the harmonic representation operators. Similar results are also obtained for the extended
harmonic representation of the semi-direct product of the $(2p+2q+1)$-dimensional
Heisenberg group by $SU(p,q)$. \end{abstract}

\vspace{1cm}

\section {Introduction} \label{sec:1}

Let $p$ and $q$ be positive integers and let $n=p+q$.
Let $H_n$ be the $(2n+1)$-dimensional Heisenberg group. Then
$SU(p,q)$ naturally acts on ${\mathbb C}^n$, hence on $H_n$. We denote by $k\cdot h$ the action of $k\in SU(p,q)$ on $h\in H_n$. Also, for each unitary irreducible representation $\rho$ of $H_n$ and each $k\in SU(p,q)$, we denote by $\rho_k$ the representation of $H_n$ defined
by $\rho_k(h)=\rho(k\cdot h)$ for each $h\in H_n$.

Now, we fix a non degenerate irreducible unitary representation $\rho$ of $H_n$ on the Fock space $\mathcal F$. Then, for each $k\in SU(p,q)$, $\rho_k$ and $\rho$
have the same central character hence are unitarily equivalent.
Thus there exists a family $\sigma(k)$, $k\in SU(p,q)$ of unitary operators on $\mathcal F$ 
such that 
\begin{equation*}\rho_k(h)\sigma(k)=\sigma(k)\rho(h) \end{equation*}
for each $k\in SU(p,q)$ and $h\in H_n$. 
Moreover, one can choose the operators $\sigma(k)$ in such a way that $\sigma$
is a unitary representation of $SU(p,q)$, see \cite{CaN, Da}. Then $\sigma$
is called the {\it harmonic representation} of $SU(p,q)$ and its irreducible components
called the {\it ladder representations} are of some importance both in Mathematics
\cite{KVe} and in Physics \cite{CR1}. Note that $\sigma$ is analogous to the so-called
metaplectic representation of $Sp(n,{\mathbb R})$ which was intensively studied, see for
instance \cite{Fo} and \cite{KVe}. More precisely, $SU(p,q)$ can be embedded in
$Sp(n,{\mathbb R})$ and then $\sigma$ can also be introduced as a restriction of the
metaplectic representation. Consequently, it is possible to study $\sigma$ starting from
known results on the metaplectic representation, see in particular \cite{KVe}. Here, however,
we prefer to consider $\sigma$ independently of the metaplectic representation as in \cite{Da}
and \cite{CaN}.
 
The present paper can be considered as a continuation of \cite{CaN}. In \cite{CaN},
we follow the method indicated in \cite[p. 533]{Ne} and obtain an explicit expression for the Berezin symbol or, equivalently, for the integral kernel of $\sigma(k)$ for $k\in SU(p,q)$, by using the highest weight representations of $H_n\rtimes SU(p,q)$. Then we get formulas which are close to those given in \cite{Da} and \cite{L}.

Here, we give explicit formulas for the complex Weyl symbol of $\sigma(k)$, $k\in SU(p,q)$.
The complex Weyl calculus $W$ can be defined as the correspondence between operators on $\mathcal F$ and functions on ${\mathbb C}^n$ obtained by translating the usual Weyl correspondence by means of the Bargmann transform, see \cite{Fo}.

Our computation of $W(\sigma(k))$ is mainly based on the results of \cite{CaN} and on basic Gaussian calculus. We also give a formula for $W(d\sigma(X))$, $X\in su(p,q)$. Note that similar
results on Weyl symbols of metaplectic representation operators were obtained by various methods, see \cite{CaComp, CR1, CR2, Ho2}. Such formulas for Weyl symbols are useful to study the quantum unitary propagators for quadratic Hamiltonians
which are important  in quantum mechanics, see \cite[Chapter 3]{CR1}.

The extended harmonic representation is the unitary representation $\pi$ of $G:=H_n\rtimes SU(p,q)$ defined by $\pi(h,k)=\rho(h)\sigma(k)$  for each $h\in H_n$ and $k\in SU(p,q)$.
Note that $\pi$ is the exact analogue of the extended metaplectic 
representation \cite[p. 196]{Fo}.
In this paper, we deduce a formula for $W(\pi(g))$, $g\in G$, from the previously established formula
for $W(\sigma(k))$, $k\in SU(p,q)$ by exploiting the covariance of $W$ with respect to
$\rho$, see \cite{CaRiv, CaTs}.

This paper is organized as follows. We first introduce the Heisenberg group, its unitary irreducible representations on $\mathcal F$ and the Berezin correspondence on $\mathcal F$ (Section \ref{sec:2}). Then, we recall the construction of the complex Weyl correspondence $W$
from a Stratonovich-Weyl quantizer defined by using $\rho$
(Section \ref{sec:3}). In Section \ref{sec:4}, we establish the first main result of the paper,
namely some formulas for $W(\sigma(k))$ (Theorem \ref{th:Wsigma}). In Section \ref{sec:5}, we introduce
the extended harmonic representation $\pi$ and give a formula for $W(\pi(g))$. This is the second main result of the paper (Theorem \ref{th:Wpi}). We also give a formula for $W(d\pi(X))$ for
$X$ in the Lie algebra of $G$. See also \cite{CaExtM}, where similar results for the extended
metaplectic representation are established by direct, but tedious, computations.

\section {Berezin quantization on the Fock space} \label{sec:2}

We first introduce the Heisenberg group.
Let $p$ and $q$ be positive integers and let $n=p+q$. For $z
=(z_1,z_2,\ldots, z_n)=\sum_{j=1}^nz_je_j \in {\mathbb C}^n$  and $w=
(w_1,w_2,\ldots, w_n) \in {\mathbb C}^n$, we write $zw:=\sum_{j=1}^n z_jw_j$ for simplicity. For $A\in M_n({\mathbb R})$, we define $Az:=zA^t$  where the subscript $t$ denotes transposition.

Let $J_0=\left(\begin{smallmatrix} I_p&0\\
0&-I_q\end{smallmatrix}\right)$ and let $\omega$ be the symplectic form on ${\mathbb C}^{n}\times {\mathbb
C}^{n}$ defined by
\begin{equation*}\omega
((z,w),(z',w'))=\tfrac{i}{2}(z(J_0w')-z'(J_0w)).
\end{equation*} for $z,w,z',w'\in {\mathbb C}^n$.

The $(2n+1)$-dimensional (real) Heisenberg group is \begin{equation*}
H_n:=\{((z,{\bar z}),c)\,:\,z\in {\mathbb C}^n, c\in {\mathbb R}\}\end{equation*}
equipped with the multiplication defined by
\begin{equation*}((z,{\bar
z}),c)\cdot ((z',{\bar z'}),c')=((z+z',{\bar z}+{\bar
z'}),c+c'+\tfrac{1}{2}\omega ((z,{\bar z}),(z',{\bar
z'}))).\end{equation*} 

For $z\in {\mathbb C}^n$, we write $z=(z^+,z^-)$ where $z^+\in {\mathbb C}^p$ and
$z^-\in {\mathbb C}^q$. Let $s: {\mathbb C}^n\rightarrow {\mathbb C}^n$ be the involution defined by $s(z)=(z^{+},\bar{z}^-)$.

The center of $H_n$ is $C(H_n):=\{((0,0),c)\,:\,c\in{\mathbb R}\}$.
Let $\lambda>0$. By the Stone-von Neumann
theorem, there exists a (unique, up to unitary equivalence) unitary
irreducible representation $\rho_{\lambda}$ of $H_n$ whose restriction to $C(H_n)$ is the character $\chi_{\lambda}:((0,0),c)\rightarrow e^{i\lambda c}$,
see \cite{Fo, Tay}. Then $\rho_{\lambda}$ is called a \textit{generic} representation of $H_n$. The general method for constructing holomorphic representations of Hermitian Lie groups developped in \cite{Ne} when applied to $H_n$ leads to introduce the Fock space, that is,
the Hilbert space
${\mathcal F}_{\lambda}$ consisting of all holomorphic functions on ${\mathbb C}^n$ such
that
\begin{equation*}\Vert f\Vert
^2:=\left(\tfrac{\lambda}{2\pi}\right)^n
\int_{{\mathbb C}^n}\,\vert f(z)\vert ^2\,
e^{-\lambda \vert z\vert ^2/2}\,dm(z)<+\infty \end{equation*}
where $dm(z):=dx\,dy$ is the Lebesgue measure on ${\mathbb C}^n$ ($z=x+iy$, $x,y \in {\mathbb R}^n$) and the realization of $\rho_{\lambda}$ on ${\mathcal F}_{\lambda}$
given by
 \begin{equation*}(\rho_{\lambda}(g)f)(z)=e^{i\lambda c_0}\,
\exp \lambda\left( \tfrac{1}{2}\overline{s(z_0)}z-\tfrac{1}{4}\vert z_0\vert^2\right)\, f(z-s(z_0))
\end{equation*}
for $g=((z_0,{\bar z_0}),c_0)\in H_n$ and $z\in {\mathbb C}^n$. See \cite{CaN} for details.

Recall that ${\mathcal F}_{\lambda}$ is a reproducing kernel Hilbert space. More precisely, for each $z\in {\mathbb C}^n$, let $e_z$ be the function of ${\mathcal F}_{\lambda}$ defined by
$e_z(w)=\exp (\lambda{\bar z}w/2)$. Then we have the reproducing property
$f(z)=\langle f,e_z\rangle $ for each $f\in {\mathcal F}_{\lambda}$.

Now, we introduce the Berezin
calculus on ${\mathcal F}_{\lambda}$ \cite{Be1, Be2, CaRiv}.
The Berezin (covariant) symbol of the operator $A$ on ${\mathcal F}_{\lambda}$ is the
function $S(A)$ defined on ${\mathbb C}^n$ by
\begin{equation*}S(A)(z):=\frac{\langle A\,e_z\,,\,e_z\rangle}{\langle e_z\,,\,e_z\rangle} \end{equation*} and the double Berezin symbol of $A$ is the function $S'(A)$ on ${\mathbb C}^n\times {\mathbb C}^n$ defined by
\begin{equation*}S'(A)(z,w):=\frac{\langle A\,e_w\,,\,e_z\rangle}{\langle e_w\,,\,e_z\rangle}. \end{equation*}

On the one hand, by the reproducing property, for each operator $A$ on  
${\mathcal F}_{\lambda}$, $f\in {\mathcal F}_{\lambda}$ and $z\in {\mathbb C}^n$,
we have
\begin{equation*}(Af)(z)=\langle Af,e_z\rangle= \langle f,A^{\ast}e_z\rangle.
\end{equation*}
This shows that the kernel $k_A(z,w)$ of $A$ is given by
\begin{equation*}k_A(z,w)=\overline{(A^{\ast}e_z)(w)}=\langle Ae_w,e_z\rangle
=S'(A)(z,w)\langle e_w,e_z\rangle.\end{equation*}

On the other hand, $S'(A)(z,w)$ is holomorphic in the variable $z$ and anti-holomorphic in the variable $w$, then $S'(A)$ is determined by its restriction to the diagonal of ${\mathbb C}^n \times {\mathbb C}^n$, that is, by $S(A)$. Consequently, $A$ is determined by $S(A)$.
We have the following result, see \cite{Be1, CaRiv, CaTs, UU}.
\begin{proposition} \label{prop:Ber}
\begin{enumerate}
\item For each operator $A$ on ${\mathcal F}_{\lambda}$ and each $z\in {\mathbb C}^n$, we
have $S(A^{\ast})(z)=\overline {S(A)(z)}$;
 
\item Consider the action of $H_n$ on ${\mathbb C}^n$ defined by
$h\cdot z =z+s(z_0)$ for $h=((z_0,\bar{z_0}),c_0)\in H_n$ and $z\in {\mathbb C}^n$. 
Then,  for each operator $A$ on ${\mathcal F}_{\lambda}$, $h\in H_n$ and 
$z\in {\mathbb C}^n$, we have

\begin{equation*}S(A)(h\cdot z)=S(\rho_{\lambda}(h)^{-1}A\rho_{\lambda}(h))(z);\end{equation*}

\item Denote by ${\mathcal L}_2({{\mathcal F}_{\lambda}})$ the space of all Hilbert-Schmidt operators on ${\mathcal F}_{\lambda}$ (endowed with the
Hilbert-Schmidt norm). Then $S$ is a bounded operator from
${\mathcal L}_2({\mathcal F}_{\lambda})$ to $L^2({\mathbb C}^n,\mu_{\lambda})$  which
is one-to-one and has dense range.
\end{enumerate} \end{proposition}

Note that the second assertion of the proposition immediatly follows from the following
relation. For each $h=((z_0,\bar{z_0}),c_0)\in H_n$ and $z\in {\mathbb C}^n$, we have
\begin{equation} \label{eq:rez} \rho_{\lambda}(h)e_z=\exp \left(i\lambda c_0-\tfrac{\lambda}{2}s(z_0){\bar z}-\tfrac{\lambda}{4}\vert z_0\vert^2\right)\,e_{h\cdot z}.
\end{equation}

Let us also recall the Berezin transform.
Let $S^{\ast}$ be the adjoint operator of $S:{\mathcal L}_2({\mathcal F}_{\lambda}) \rightarrow L^2({\mathbb C}^n,m_{\lambda})$ where $m_{\lambda}$ is the measure on ${\mathbb C}^n$ defined by $dm_{\lambda}(z)=({2\pi})^{-n}{\lambda}^n dm(z)$. Then
the Berezin transform is the operator $B$
on $L^2({\mathbb C}^n,m_{\lambda})$ defined by $B:=SS^{\ast}$.
Hence we have 
the integral formula
\begin{equation*}(Bf)(z)=\int_{{\mathbb C}^n}\,f(w)
\,e^{ -\lambda\vert z-w\vert^2/2}\,dm_{\lambda}(w),\end{equation*}
see \cite{Be1, Be2, UU}.
It is also known that $B=\exp (\Delta/2\lambda)$ where
$\Delta=4\sum_{k=1}^n\partial^2/\partial z_k\partial {\bar z}_k$,
see \cite{Luo, UU}.

\section{Complex Weyl correspondence on the Fock space} \label{sec:3}

As the classical Weyl correspondence, the complex Weyl correspondence can be constructed from a \textit{Stratonovich-Weyl quantizer}, see \cite{AU1, GB}. Here we follow more or less the exposition of \cite{CaTs}, see also \cite{CaTo}.

We start from the parity operator $R$ on ${\mathcal F}_{\lambda}$ defined by
\begin{equation*}(Rf)(z)=2^nf(-z). \end{equation*}

The map $z\rightarrow h_z:=((s(z),\overline{s(z)}),0)$ is a section
for the action of $H_n$ on ${\mathbb C}^n$, that is, we have $h_z \cdot 0=z$
for each $z\in {\mathbb C}^n$.
Let us define, for each $z\in {\mathbb C}^n$,
\begin{equation*}\Omega(z):=\rho_{\lambda}(h_z)R\rho_{\lambda}(h_z)^{-1}. \end{equation*} 

By an easy computation, we get
\begin{equation*}
(\Omega(z)f)(w)=2^n\exp \left(\lambda ({\bar z}w-\vert z\vert^2)\right) f(2z-w) \end{equation*}
for each $z, w\in {\mathbb C}^n$. 
The map $\Omega$ is called a \textit{Stratonovich-Weyl quantizer} and satisfies the covariance property
\begin{equation*}\Omega(h\cdot z)=\rho_{\lambda}(h)\Omega(z)\rho_{\lambda}(h)^{-1} \end{equation*} for each $h\in H_n$ and $z \in {\mathbb C}^n$.

For each trace-class operator $A$ on ${\mathcal F}_{\lambda}$, we define the function $W(A)$ on ${\mathbb C}^n$  by
\begin{equation*} W(A)(z):=\Tr(A\Omega(z)).\end{equation*}

\begin{proposition} \label{prop:intW} \rm{\cite{CaTo}} For each trace-class operator $A$ on
${\mathcal F}_{\lambda}$
and each $z \in {\mathbb C}^n$, we have
\begin{equation} \label{eq:intW}
W(A)(z)=2^n\int_{{\mathbb C}^n}k_A(z+w,z-w)\exp \left(\tfrac{\lambda}{2}\left(-z{\bar z}-w{\bar w}+z{\bar w}-{\bar z}w\right)\right) dm_{\lambda}(w).\end{equation}\end{proposition}

We can use this integral formula in order to define $W(A)$ for more general operators
on ${\mathcal F}_{\lambda}$  as for instance Hilbert-Schmidt operators or even differential
operators with polynomial coefficients, see \cite{AU1}, \cite[Proposition 4]{CaTo}.

\begin{proposition} \label{prop:unitary} The map $W:{\mathcal L}_2({\mathcal F}_{\lambda})\rightarrow L^2({\mathbb C}^n,m_{\lambda})$ is the unitary part in the polar decomposition of $S$, that is we have $S=(SS^{\ast})^{1/2}W=B^{1/2}W$.\end{proposition}

Various proofs of this proposition can be found in \cite{OZ, Luo, CaRiv} and \cite{CaTs}.
Note that, in \cite{OZ}, the space ${\mathcal L}_2({\mathcal F}_{\lambda})$ is identified
with the tensor product ${\mathcal F}_{\lambda}\otimes \overline {{\mathcal F}_{\lambda}}$
of ${\mathcal F}_{\lambda}$ with its conjugate. Then $S$ can be seen as the restriction map
\begin{equation*}f(z,w) \rightarrow f(z,z)\exp (-\lambda \vert z\vert^2/2). \end{equation*}
Note that the Weyl correspondence introduced in \cite{OZ} is the inverse map of $W$.

An alternative method is to define $W$ (on Hilbert-Schmidt operators) as 
the unitary part in the polar decomposition of $S:{\mathcal L}_2({\mathcal F}_{\lambda})\rightarrow L^2({\mathbb C}^n,m_{\lambda})$ and then to obtain Eq. \ref{eq:intW}
by using the properties of $B$, see Section \ref{sec:2}.

Now, we can easily verify the following proposition which will be used later.

\begin{proposition} \label{prop:covWrho} \rm{\cite{CaTo}} 
The map $W$ is covariant with respect to $\rho_{\lambda}$, that is,
for each operator $A$ on ${\mathcal F}_{\lambda}$, $h\in H_n$ and 
$z\in {\mathbb C}^n$, we have
\begin{equation*}W(A)(h\cdot z)=W(\rho_{\lambda}(h)^{-1}A\rho_{\lambda}(h))(z).
\end{equation*} \end{proposition}  

Various generalizations of the complex Weyl correspondence can be found in \cite{AU1,AU2,GB}.

\section{Complex Weyl symbols of harmonic representation operators} \label{sec:4}

Here we retain the notation of Sect. \ref{sec:2}. We write the elements $g$ of $SU(p,q)$
as block matrices $g=\left(\begin{smallmatrix} A&B\\
C&D\end{smallmatrix}\right)$ where $A,B,C,D$ are $p\times p$, $p\times q$, $q\times p$, $q\times q$
matrices satisfying the relations
\begin{align*}&AA^{\star}-BB^{\star}=I_p,\,\, CC^{\star}-DD^{\star}=-I_q,\\
&A^{\star}A-C^{\star}C=I_p,\,\,B^{\star}B-D^{\star}D=-I_q,\\
&A^{\star}B=C^{\star}D, \,\,BD^{\star}=AC^{\star}.
\end{align*}
Here the subscript $\star$ denotes the conjugate transpose.

The group $SU(p,q)$ naturally acts on ${\mathbb C}^n$ hence acts on $H_n$  by
\begin{equation*}k\cdot ((z,\bar{z}),c)=((kz,\overline{kz}),c).\end{equation*}

We fix $\lambda>0$. For simplicity, we denote $\rho:=\rho_{\lambda}$ and ${\mathcal F}=
{\mathcal F}_{\lambda}$.

For each $k\in SU(p,q)$, we define $\rho_k$ by $\rho_k(h):=\rho(k\cdot h)$. Then, since $\omega$ is $SU(p,q)$-invariant, that is, we have
\begin{equation*}\omega(k(z,\bar{z}),k(z',\bar{z'}))=\omega((z,\bar{z}),(z',\bar{z'}))\end{equation*} for each $z, z'\in {\mathbb C}^n$, we can verify that $\rho_k$
is also a generic representation of $H_n$, which has the same central character as $\rho$. 

In \cite{CaN}, we presented a proof of the following result which gives explicit formulas for the kernels of the harmonic representation operators. This theorem, together with Proposition
\ref{prop:intW}, is the main tool for the computations of the complex Weyl symbols of the
harmonic representation operators. 

\begin{theorem} \label{th:rappel} For each $k=\left(\begin{smallmatrix}A&B\\C&D\end{smallmatrix}\right)\in SU(p,q)$, let $\sigma(k)$ be the operator on $\mathcal F$ with kernel
\begin{equation*}
b_k(z,w)=(\Det A)^{-1} \exp \tfrac{\lambda}{2}\left( -({\overline D}^{-1}{\overline C}\bar{w_1})\bar{w_2}+({\overline D}^{-1}z_2)\bar{w_2}+(A^{-1}z_1)\bar{w_1}+
z_1({\overline B}({\overline D})^{-1}z_2) \right)\end{equation*}
where $z=(z_1,z_2), w=(w_1,w_2)$ with $z_1,w_1\in {\mathbb C}^p$ and 
$z_2,w_2\in {\mathbb C}^q$. Then 
\begin{enumerate} \item $\sigma$ is a unitary representation of $SU(p,q)$;
\item For each $h\in H_n$ and each $k\in SU(p,q)$, we have $\rho_k(h)\sigma(k)=
\sigma(k)\rho(h)$;
\item For each $k=\left(\begin{smallmatrix}A&B\\C&D\end{smallmatrix}\right)\in SU(p,q)$
and each $z=(z_1,z_2)\in {\mathbb C}^p \times {\mathbb C}^q$, we have
\begin{align*}
S(\sigma(k))&(z)=(\Det A)^{-1}\\
\times & \exp \tfrac{\lambda}{2}\left(-\vert z\vert^2 -({\overline D}^{-1}{\overline C}\bar{z_1})\bar{z_2}+({\overline D}^{-1}z_2)\bar{z_2}+(A^{-1}z_1)\bar{z_1}+
z_1({\overline B}({\overline D})^{-1}z_2) \right).\end{align*}
\item Let $X=\left(\begin{smallmatrix}A&B\\C&D\end{smallmatrix}\right)\in su(p,q)$.
Then, for each $z=(z_1,z_2)\in {\mathbb C}^p \times {\mathbb C}^q$, we have 
\begin{equation*}S(d\sigma(X))(z)=-\Tr(A)-\tfrac{\lambda}{2}\bigl(({\overline C}\bar{z_1})\bar{z_2}+({\overline D}z_2)\bar{z_2}+(A z_1)\bar{z_1}-z_1({\overline B}z_2)
\bigr).\end{equation*} 
\end{enumerate}\end{theorem}

In fact, we can also derive Theorem \ref{th:rappel} from results on the metaplectic representation. 
In particular,  we can recover (up to normalization) Theorem \ref{th:rappel}, (1) from \cite[Eq. (4.38)]{Fo} by using
the embedding
\begin{equation*}k=\begin{pmatrix}A&B\\C&D\end{pmatrix}\rightarrow
{\tilde k}=\begin{pmatrix}A&0&0&B\\
0&{\overline D}&{\overline C}&0\\
0&{\overline B}&{\overline A}&0\\
C&0&0&D\\
\end{pmatrix} \end{equation*}
from $SU(p,q)$ to a subgroup of $SU(n,n)$ which is isomorphic to $Sp(n, {\mathbb R})$.

Note that there are many ways to construct the metaplectic representation. Let us mention some of them.

In \cite{SW}, Sternberg and Wolf used representations of Lie superalgebras to construct the metaplectic representation of $Sp(n, {\mathbb R})$ and also studied its restriction
to $U(p,q)$, $p+q=n$. Formulas (4.11) and (4.12) in  \cite{SW} then correspond to  
Theorem \ref{th:rappel}, (4).

In \cite{Pe1}, in order to understand some invariance properties of theta functions, Peetre 
introduced a new object called the \textit{Fock bundle} which is a Hilbert bundle over the
Siegel's generalized upper halfplane whose fibers are Fock spaces. This also led to a construction of  the metaplectic representation, see the last formulas in \cite[Sect. 5]{Pe1}.
See also \cite{Pe2} for direct computations in the case $n=2$.

Blattner and Rawnsley \cite{BR} used geometrical quantization techniques to introduce and
study the harmonic representation.

In \cite{Da}, a formula for the kernels of the harmonic representation operators
was obtained by group-theorical methods. This formula is similar to that of 
Theorem \ref{th:rappel}, (1).

More generally, Howe \cite{How} took a higher point of view by introducing and studying
the \textit{oscillator semigroup} which is a semigroup of contractions whose closure contains
the metaplectic representation operators.

Now we aim to compute $W(\sigma(k))$ for $k\in SU(p,q)$. This computation is based on the following lemmas.

\begin{lemma} \label{lem:prep1} Let $R,S,T\in M_n({\mathbb C})$ such that $R$ and $T$
are symmetric. Assume that $\bigl(\begin{smallmatrix} R & S\\ S^t & T \end{smallmatrix}\bigr)$ is invertible and denote by $\bigl(\begin{smallmatrix} \alpha &\beta \\ \beta^t & \gamma \end{smallmatrix}\bigr)$ its inverse. Then we have 
\begin{equation*}
\begin{pmatrix} -R&2I_n-S\\
S^t-2I_n& T\end{pmatrix}
\begin{pmatrix} \alpha &\beta \\ \beta^t & \gamma
\end{pmatrix}\begin{pmatrix} -R&S-2I_n\\
2I_n-S^t& T\end{pmatrix}=
\begin{pmatrix} R+4\gamma &4I_n-S-4\beta^t\\
4I_n-S^t-4\beta& T+4\alpha\end{pmatrix}. \end{equation*}
\end{lemma}

\begin{proof} By writing
\begin{equation*}\begin{pmatrix} \alpha &\beta \\ \beta^t & \gamma
\end{pmatrix}\begin{pmatrix} R & S\\ S^t & T
\end{pmatrix}=I_{2n}\end{equation*}
we get the equations
\begin{align*} &\alpha R+\beta S^t=I_n; \quad \alpha S+\beta T=0;\\
&\beta^t R+\gamma S^t=0;\quad \beta^t S+\gamma T=I_n,\end{align*}
plus four additional equations obtained by transposition. By using all these equations
we obtain
\begin{equation*}\begin{pmatrix} -R&2I_n-S\\
S^t-2I_n& T\end{pmatrix}
\begin{pmatrix} \alpha &\beta \\ \beta^t & \gamma
\end{pmatrix}=\begin{pmatrix} 2\beta^t -I_n&2\gamma\\
-2\alpha& I_n-2\beta\end{pmatrix}
\end{equation*} and
\begin{equation*}\begin{pmatrix} 2\beta^t -I_n&2\gamma\\
-2\alpha& I_n-2\beta\end{pmatrix}\begin{pmatrix} -R&S-2I_n\\
2I_n-S^t& T\end{pmatrix}=
\begin{pmatrix} R+4\gamma &4I_n-S-4\beta^t\\
4I_n-S^t-4\beta& T+4\alpha\end{pmatrix}. \end{equation*}
\end{proof}

\begin{lemma} \label{lem:prep2} Let $k=\left(\begin{smallmatrix}A&B\\C&D\end{smallmatrix}\right)\in SU(p,q)$. Let $P=\left(\begin{smallmatrix}A&0\\0&{\overline D}\end{smallmatrix}\right)$,  $Q=\left(\begin{smallmatrix}0&B\\{\overline C}&0\end{smallmatrix}\right)$ and
$\tilde{k}=\left(\begin{smallmatrix}P&Q\\{\overline Q}&{\overline P}\end{smallmatrix}\right)$.
Moreover, with the notation as in the preceding lemma, take
\begin{equation*}R=\begin{pmatrix}0&-{\overline B}{\overline D}^{-1}\\
-(D^{\star})^{-1}B^{\star}& 0\end{pmatrix}; S=\begin{pmatrix}I_p+(A^t)^{-1}&0\\
0& I_q+(D^{\star})^{-1}\end{pmatrix};  T=\begin{pmatrix}0&C^{\star}(D^{\star})^{-1}\\
{\overline D}^{-1}{\overline C}& 0\end{pmatrix}\end{equation*} and let
\begin{equation*}M=\begin{pmatrix}R&S\\S^t&T\end{pmatrix}.\end{equation*} 
Then 
\begin{enumerate}\item We have 
\begin{equation*}
\begin{pmatrix}0&P\\I_n&{\overline Q}\end{pmatrix}
M=\tilde{k}+I_{2n};\end{equation*}
\item We have $\Det(M)=(-1)^n(\Det(A)\Det({\overline D}))^{-1}\Det(\tilde{k}+I_{2n})$;
\item Assume that $\Det(k+I_n)\not=0$ and denote $J=\left(\begin{smallmatrix}0&I_n\\-I_n&0\end{smallmatrix}\right)$. We have
\begin{equation*}\begin{pmatrix} \gamma &\tfrac{1}{2}-\beta^t \\ \tfrac{1}{2}-\beta & \alpha\end{pmatrix}=\tfrac{1}{2}J(\tilde{k}+I_{2n})^{-1}(\tilde{k}-I_{2n}).\end{equation*}
\end{enumerate}
\end{lemma}

\begin{proof} (1) We can verify that
\begin{equation*}\begin{pmatrix}0&P\\I_n&{\overline Q}\end{pmatrix}
\begin{pmatrix} R & S\\ S^t & T \end{pmatrix}=
\begin{pmatrix} PS^t & PT\\ R+{\overline Q}T &S+{\overline Q} T \end{pmatrix}=
\begin{pmatrix}P+I_n&Q\\{\overline Q}&{\overline P}+I_n\end{pmatrix}=\tilde{k}+I_{2n}
\end{equation*}
by using the relations given at the beginning of this section. For instance, we have
\begin{equation*}S+{\overline Q} T=\begin{pmatrix} I_n+(A^t)^{-1}+{\overline B}{\overline D}^{-1}{\overline C} & 0\\ 0& I_n+(D^{\star})^{-1}+CC^{\star}(D^{\star})^{-1} \end{pmatrix}=\begin{pmatrix} I_p+{\overline A} & 0\\ 0 & I_q+D \end{pmatrix}
\end{equation*}
since
\begin{align*}(A^t)^{-1}+&{\overline B}{\overline D}^{-1}{\overline C}=
(A^t)^{-1}(I_p+A^t{\overline B}{\overline D}^{-1}{\overline C})=
(A^t)^{-1}(I_p+(C^t{\overline D})({\overline D}^{-1}{\overline C}))\\=&
(A^t)^{-1}(I_p+C^t{\overline C})=(A^t)^{-1} (I_p+\overline{C^{\star}C)})=
(A^t)^{-1}\overline {A^{\star}A}={\overline A}.
\end{align*}

(2) The result is obtained by taking  determinants of both sides of the matrix equation
proved in (1).

(3) First note that $\Det(\tilde{k}+I_{2n})=\vert \Det(k+I_{n})\vert^2$. Let $J_1=
\bigl(\begin{smallmatrix} 0& I_n\\ I_n & 0 \end{smallmatrix}\bigr)$. Taking (1) into account,
we have

\begin{align*}\begin{pmatrix}\gamma &\tfrac{1}{2}I_n-\beta^t\\
\tfrac{1}{2}I_n-\beta&\alpha \end{pmatrix}&=\tfrac{1}{2}J_1+
\begin{pmatrix} \gamma &-\beta^t \\-\beta & \alpha  \end{pmatrix}\\
&=\tfrac{1}{2}J_1-J\begin{pmatrix} \alpha &\beta \\
\beta^t & \gamma  \end{pmatrix}J\\
&=\tfrac{1}{2}J_1-J(\tilde {k}+I_{2n})^{-1}\begin{pmatrix} 0&P\\
I_n&{\overline Q} \end{pmatrix}J\\
&=J(\tilde {k}+I_{2n})^{-1}\left( \tfrac{1}{2}(\tilde {k}+I_{2n})J^{-1}J_1-\begin{pmatrix} 0&P\\
I_n&{\overline Q} \end{pmatrix}J\right)\\
&=\tfrac{1}{2}J(\tilde {k}+I_{2n})^{-1}(\tilde {k}-I_{2n}).
\end{align*}
\end{proof}

Let \begin{equation*}U=\begin{pmatrix} I_p&iI_p&0&0\\
0&0&I_q&iI_q\\
I_p&-iI_p&0&0\\
0&0&I_q&-iI_q\\ \end{pmatrix}.\end{equation*}
Then we have $UU^{\star}=2I_{2n}$ and $U U^t=2J_1$.

\begin{lemma} \label{lem:prep3} Assume that $\Det(k+I_{n})\not=0$. Retain the notation of Lemma \ref{lem:prep2}. Then $\Rea(U^tMU)$ is a positive definite matrix.
\end{lemma}

\begin{proof} By using the relations introduced at the beginning of this section, we can verify that the matrix
\begin{equation*}M':=M-J_1=\begin{pmatrix} R & S-I_n\\ S^t-I_n & T \end{pmatrix}
=\begin{pmatrix}0&-{\overline B}{\overline D}^{-1}&(A^t)^{-1}&0\\
-(D^{\star})^{-1}B^{\star}&0&0&(D^{\star})^{-1}\\
A^{-1}&0&0&C^{\star}(D^{\star})^{-1}\\
0&{\overline D}^{-1}&{\overline D}^{-1}{\overline C}&0\\
\end{pmatrix}\end{equation*}
is unitary hence $J_1M'$ is unitary. Then we can write
\begin{equation*}J_1M'=U_0
\Diag(e^{it_1},e^{it_2},\ldots,e^{it_{2n}})U_0^{\star}\end{equation*}
where $U_0$ is a unitary matrix  and
$t_1,t_2,\ldots,t_{2n}$ are real numbers.

Let $N=U^tMU$. Then
\begin{align*}UNU^{\star}&=4J_1M=4J_1(M'+J_1)=4(J_1M'+I_{2n})\\
&=4U_0\Diag(1+e^{it_1},1+e^{it_2},\ldots,1+e^{it_{2n}})U_0^{\star}.
\end{align*}
Thus we can write
\begin{equation*}N=U_1
\Diag(1+e^{it_1},1+e^{it_2},\ldots,1+e^{it_{2n}})U_1^{\star}\end{equation*}
where $U_1$ is an invertible matrix. Now, since $\Det(N)=(\Det(U))^2\Det(M)\not= 0$ by Lemma \ref{lem:prep2}, we have $1+e^{it_j}\not= 0$ hence $1+\cos(t_j)>0$ for each $j=1,2,\ldots,2n$. But
\begin{equation*}\Rea(N)=\tfrac{1}{2}(N+{\overline N})=\tfrac{1}{2}(N+N^{\star})=U_1\Diag(1+\cos(t_1),1+\cos(t_2),\ldots,1+\cos(t_{2n}))U_1^{\star}.
\end{equation*} Consequently, $\Rea(N)$ is a positive definite matrix.
\end{proof}

\begin{theorem} \label{th:Wsigma} Let $k=\left(\begin{smallmatrix}A&B\\C&D\end{smallmatrix}\right)\in SU(p,q)$. Assume that  $\Det(k+I_{n})\not=0$.
Then, with the notation of the preceding lemmas, for each $z\in {\mathbb C}^n$, we have
\begin{align*}W(\sigma(k))&(z)=2^n\vert \Det(k+I_{n})\vert^{-1}
\exp \left(\lambda \begin{pmatrix}z&
{\bar z}\end{pmatrix}\begin{pmatrix}\gamma &\tfrac{1}{2}I_n-\beta^t\\
\tfrac{1}{2}I_n-\beta&\alpha \end{pmatrix}\begin{pmatrix}z\\{\bar z}\end{pmatrix}\right)\\
=&
2^n\vert \Det(k+I_{n})\vert^{-1}
\exp \left( \tfrac{\lambda}{2}\begin{pmatrix}z&
{\bar z}\end{pmatrix}J
({\tilde k}+I_{2n})^{-1}({\tilde k}-I_{2n})\begin{pmatrix}z\\{\bar z}\end{pmatrix}\right).
\end{align*}
\end{theorem} 

\begin{proof} Let $k=\left(\begin{smallmatrix}A&B\\C&D\end{smallmatrix}\right)\in SU(p,q)$ such that  $\Det(k+I_{n})\not=0$. We have
\begin{equation*}W(\sigma(k))(z)=\left(\tfrac{\lambda}{\pi}\right)^n\int_{{\mathbb C}^n}b_k(z+w,z-w)\exp \left(\tfrac{\lambda}{2}\left(-z{\bar z}-w{\bar w}+z{\bar w}-{\bar z}w\right)\right) dm(w)\end{equation*}
where the kernel $b_k$ of $\sigma(k)$ is given by Theorem \ref{th:rappel}. Then we obtain,
for each $z=(z_1,z_2)\in {\mathbb C}^p\times {\mathbb C}^q$,
\begin{align*}W&(\sigma(k))(z)=\left(\tfrac{\lambda}{\pi}\right)^n
(\Det A)^{-1}\\
& \times \exp \left(\tfrac{\lambda}{2}\left(
-({\overline D}^{-1}{\overline C}{\bar z_1}){\bar z_2}+({\overline D}^{-1}z_2){\bar z_2}+(A^{-1}z_1){\bar z_1}+z_1({\overline B}{\overline D}^{-1}{\bar z_2})-z{\bar z}\right)\right)
\, I(k)\end{align*}
where 
\begin{align*}I&(k):=\int_{{\mathbb C}^n}
\exp \left(\tfrac{\lambda}{2}\left(
-({\overline D}^{-1}{\overline C}{\bar w_1}){\bar w_2}-({\overline D}^{-1}w_2){\bar w_2}-(A^{-1}w_1){\bar w_1}+w_1({\overline B}{\overline D}^{-1}{ w_2})-w{\bar w}\right)\right)\\
\times & \exp \Bigl(\tfrac{\lambda}{2}\bigl(
({\overline D}^{-1}{\overline C}{\bar z_1}){\bar w_2}+
({\overline D}^{-1}{\overline C}{\bar w_1}){\bar z_2}
-({\overline D}^{-1}z_2){\bar w_2}+({\overline D}^{-1}w_2){\bar z_2}\\
-&(A^{-1}z_1){\bar w_1}+(A^{-1}w_1){\bar z_1}+z_1({\overline B}{\overline D}^{-1}{ w_2})+w_1({\overline B}{\overline D}^{-1}{z_2})+z{\bar w}-{\bar z}w \bigr)\Bigr)\,dm(w).
\end{align*}
Let
\begin{equation*}M=\begin{pmatrix}0&-{\overline B}{\overline D}^{-1}&I_p+(A^t)^{-1}&0\\
-(D^{\star})^{-1}B^{\star}&0&0&I_q+(D^{\star})^{-1}\\
I_p+A^{-1}&0&0&C^{\star}(D^{\star})^{-1}\\
0&I_q+{\overline D}^{-1}&{\overline D}^{-1}{\overline C}&0\\
\end{pmatrix}\end{equation*} 
and
\begin{align*}V&=\bigl( (A^t)^{-1}-I_p){\bar z_1}+{\overline B}{\overline D}^{-1}z_2,
((D^{\star})^{-1}-I_q){\bar z_2}+(D^{\star})^{-1}B^{\star}z_1,\\
& (I_p-A^{-1})z_1+C^{\star}(D^{\star})^{-1}{\bar z_2}, (I_q-{\overline D}^{-1})z_2+
{\overline D}^{-1}{\overline C}{\bar z_1}\bigr).
\end{align*}
Then we have
\begin{align*}I(k)&=\int_{{\mathbb C}^n}\exp \left(-\tfrac{\lambda}{4}
\begin{pmatrix}w_1&w_2&{\bar w_1}&{\bar w_2}\end{pmatrix}M
\begin{pmatrix}w_1&w_2&{\bar w_1}&{\bar w_2}\end{pmatrix}^t\right)\\
\times &
\exp \left(\tfrac{\lambda}{2}V\begin{pmatrix}w_1&w_2&{\bar w_1}&{\bar w_2}\end{pmatrix}^t\right)\,dm(w)\\
&=\int_{{\mathbb R}^{2n}}
\exp \left(-\tfrac{\lambda}{4}
\begin{pmatrix}x_1&y_1&x_2&y_2\end{pmatrix}U^tMU
\begin{pmatrix}x_1&y_1&x_2&y_2\end{pmatrix}^t\right)\\
\times &
\exp \left(\tfrac{\lambda}{2}VU\begin{pmatrix}x_1&y_1&x_2&y_2\end{pmatrix}^t\right)\,
dx_1dy_1dx_2dy_2
\end{align*}
by the change of variables $w_1=x_1+iy_1$, $w_2=x_2+iy_2$, $x_1,y_1\in {\mathbb R}^p$,
$x_2,y_2\in {\mathbb R}^q$ which is also given by
\begin{equation*}
\begin{pmatrix}w_1&w_2&{\bar w_1}&{\bar w_2}\end{pmatrix}^t=U\begin{pmatrix}x_1&y_1&x_2&y_2\end{pmatrix}^t\end{equation*}
where $U$ has been defined just before Lemma \ref{lem:prep3}.

Now, recall the well-known identity for Gaussian integrals
\begin{equation*}\int_{{\mathbb R}^N}\exp (-x{\mathcal A}x+zx)\,dx=(\Det {\mathcal A})^{-1/2}\pi ^{N/2}
\exp\left(\tfrac{1}{4}z({\mathcal A}^{-1}z)\right)
\end{equation*} for $z\in {\mathbb C}^n$ and ${\mathcal A}$ a $(N\times N)$-symmetric complex matrix such that $\Rea({\mathcal A})$ is positive definite.

By Lemma \ref{lem:prep3}, this identity can be applied to $I(k)$. This gives
\begin{equation*}I(k)=\pi^n \left(\tfrac{4}{\lambda}\right)^n (\Det(U^tMU))^{-1/2}
\exp \left( \tfrac{\lambda}{4}VM^{-1}V^t\right).
\end{equation*}
But
\begin{equation*}V^t=\begin{pmatrix}0&
{\overline B}{\overline D}^{-1}&(A^t)^{-1}-I_p&0\\
(D^{\star})^{-1}B^{\star}&0&0&(D^{\star})^{-1}-I_q\\
I_p-A^{-1}&0&0&C^{\star}(D^{\star})^{-1}\\
0&I_q-{\overline D}^{-1}&{\overline D}^{-1}{\overline C}&0\\
\end{pmatrix}
\begin{pmatrix}z_1\\z_2\\ {\bar z_1}\\{\bar z_2}\end{pmatrix},
\end{equation*}
then we get
\begin{equation*}VM^{-1}V^t=\begin{pmatrix}z&
{\bar z}\end{pmatrix}
\begin{pmatrix} -R&2I_n-S\\
S^t-2I_n& T\end{pmatrix}
\begin{pmatrix} R &S\\ S^t & T
\end{pmatrix}^{-1}\begin{pmatrix} -R&S-2I_n\\
2I_n-S^t& T\end{pmatrix}
\begin{pmatrix}z\\
{\bar z}\end{pmatrix}
\end{equation*}
where, as in Lemma \ref{lem:prep2}, we denote
\begin{equation*}R=\begin{pmatrix}0&-{\overline B}{\overline D}^{-1}\\
-(D^{\star})^{-1}B^{\star}& 0\end{pmatrix}; S=\begin{pmatrix}I_p+(A^t)^{-1}&0\\
0& I_q+(D^{\star})^{-1}\end{pmatrix};  T=\begin{pmatrix}0&C^{\star}(D^{\star})^{-1}\\
{\overline D}^{-1}{\overline C}& 0\end{pmatrix}.\end{equation*}
By Lemma \ref{lem:prep1}, we thus obtain
\begin{equation*}VM^{-1}V^t=\begin{pmatrix}z&
{\bar z}\end{pmatrix}\begin{pmatrix} R+4\gamma &4I_n-S-4\beta^t\\
4I_n-S^t-4\beta& T+4\alpha\end{pmatrix}
\begin{pmatrix}z\\
{\bar z}\end{pmatrix},
\end{equation*} 
hence
\begin{equation*}I(k)=\pi^n \left(\tfrac{4}{\lambda}\right)^n (\Det(U^tMU))^{-1/2}
\exp\left(\tfrac{\lambda}{4}\begin{pmatrix}z&
{\bar z}\end{pmatrix}\begin{pmatrix} R+4\gamma &4I_n-S-4\beta^t\\
4I_n-S^t-4\beta& T+4\alpha\end{pmatrix}
\begin{pmatrix}z\\{\bar z}\end{pmatrix}\right).
\end{equation*} 
Returning to the computation of $W(\sigma(k))$, we get
\begin{align*}W&(\sigma(k))(z)=4^n(\Det (A))^{-1}\Det(U^tMU)^{-1/2}\\
\times 
&\exp\left(\tfrac{\lambda}{4}\begin{pmatrix}z&
{\bar z}\end{pmatrix}\begin{pmatrix} -R &S-2I_n\\
S^t-2I_n& -T\end{pmatrix}
\begin{pmatrix}z\\
{\bar z}\end{pmatrix}\right)\\
\times &
\exp\left(\tfrac{\lambda}{4}\begin{pmatrix}z&
{\bar z}\end{pmatrix}\begin{pmatrix} R+4\gamma &4I_n-S-4\beta^t\\
4I_n-S^t-4\beta& T+4\alpha\end{pmatrix}
\begin{pmatrix}z\\
{\bar z}\end{pmatrix}\right).
\end{align*}
By Lemma \ref{lem:prep2} again,
\begin{align*}\Det(U^tMU)&=2^{2n}(\Det (A)\Det({\overline D}))^{-1}
\Det({\tilde k}+I_{2n})\\
&=2^{2n}((\Det (A))^{-2}\vert \Det( k+I_{2n})\vert^2,
\end{align*}
since $\Det (D)=\overline { \Det (A)}$, see \cite{CaN}.

Finally, we find
\begin{equation*}W(\sigma(k))(z)=2^n\vert \Det(k+I_{n})\vert^{-1}
\exp \left(\lambda \begin{pmatrix}z&
{\bar z}\end{pmatrix}\begin{pmatrix}\gamma &\tfrac{1}{2}I_n-\beta^t\\
\tfrac{1}{2}I_n-\beta&\alpha \end{pmatrix}\begin{pmatrix}z\\{\bar z}\end{pmatrix}\right),
\end{equation*} 
as announced.
The second expression for $W(\sigma(k))(z)$ follows from (3) of Lemma \ref{lem:prep2}.
\end{proof}
We can also give a formula for $W(d\sigma(X))$, $X\in su(p,q)$.
By differentiating the map $k\rightarrow W(\sigma(k))$, we can easily verify the following result.

\begin{proposition}\label{prop:Wdsigma} Let $X=\left(\begin{smallmatrix}A&B\\C&D\end{smallmatrix}\right)\in su(p,q)$. Then, for each $z=(z_1,z_2)\in {\mathbb C}^p \times {\mathbb C}^q$, we have
\begin{equation*}W(d\sigma(X))(z)=-\tfrac{\lambda}{2}\bigl(
({\overline C}\bar{z_1})\bar{z_2}+({\overline D}z_2)\bar{z_2}\\
+(A z_1)\bar{z_1}-z_1({\overline B}z_2)\bigr).
\end{equation*}
\end{proposition}

\section{The extended harmonic representation} \label{sec:5}

We retain the notation from Section \ref{sec:4}. We consider the semi-direct product
$G:=H_n\rtimes SU(p,q)$ with respect to the action of $SU(p,q)$ on $H_n$ introduced in Section \ref{sec:4}. The elements of $G$ can be written as $((z,{\bar z}),c,k)$ where
$z\in {\mathbb C}^n$, $c\in \mathbb R$ and $k\in SU(p,q)$ and the multiplication of $G$ is given by
\begin{equation*}((z,{\bar
z}),c,k)\cdot ((z',{\bar z'}),c',k')=((z,{\bar z})+k(z',{\bar
z'}),c+c'+\tfrac{1}{2}\omega ((z,{\bar z}),k(z',{\bar
z'})),kk').\end{equation*}

By analogy with the extended metaplectic representation, \cite[p. 196]{Fo}, we define
the extended harmonic representation to be the unitary representation $\pi$ of $G$ on $\mathcal F$ defined by $\pi(h,k)=\rho(h)\sigma(k)$. The fact that $\pi$ is actually a representation is an immediate consequence of the commutation relation $\rho_k(h)\sigma(h)=
\sigma(k)\rho(h)$ for $h
\in H_n$ and $k\in SU(p,q)$, see Theorem \ref{th:rappel}. Moreover, $\pi$ is irreducible since
$\rho$ is.

Now we give a formula for $S(\pi(g))$, $g\in G$.

\begin{proposition} \label{prop:Spi}
Let $k=\left(\begin{smallmatrix}A&B\\C&D\end{smallmatrix}\right)\in SU(p,q)$ and $h=
((z_0,{\bar z_0}),c_0)\in H_n$. Let $g=(h,k)\in G$. Then, for each $z=(z_1,z_2)\in {\mathbb C}^p \times {\mathbb C}^q$, we have
\begin{align*}S(&\pi(g))(z)=(\Det (A))^{-1}\exp \left( i\lambda c_0-\tfrac{\lambda}{4}\vert z\vert ^2+\tfrac{\lambda}{2}z\overline {s(z_0)}\right) \\
&\times \exp \left( \tfrac{\lambda}{2}\left(
-({\overline D}^{-1}{\overline C}{\bar z_1}){\bar z_2}+({\overline D}^{-1}z_2){\bar z_2}+(A^{-1}z_1){\bar z_1}+z_1({\overline B}{\overline D}^{-1} z_2)\right)\right)\\
&\times \exp \left( \tfrac{\lambda}{2}\left(z_0^+({\overline B}{\overline D}^{-1}{\bar z^-_0})
-({\overline D}^{-1}{\bar z^-_0}){\bar z_2}-(A^{-1}z_0^+){\bar z_1}
-z_0^+({\overline B}{\overline D}^{-1} z_2)-z_1({\overline B}{\overline D}^{-1}{\bar z_0^-})\right)\right).
\end{align*}
\end{proposition}

\begin{proof} Let $k=\left(\begin{smallmatrix}A&B\\C&D\end{smallmatrix}\right)\in SU(p,q)$ and $h=((z_0,{\bar z_0}),c_0)\in H_n$. By using Equation  \ref{eq:rez}, we get
\begin{align*}\langle \pi(h,k)&e_z,e_z \rangle=\langle \sigma(k)e_z,\rho(h)^{-1}e_z \rangle\\
=&\exp \left(i\lambda c_0-\tfrac{\lambda}{4}\vert z_0\vert^2+\tfrac{\lambda}{2}z \overline { s(z_0)}\right)\langle \sigma(k)e_z,e_{z-s(z_0)} \rangle\\
=&\exp \left(i\lambda c_0-\tfrac{\lambda}{4}\vert z_0\vert^2+\tfrac{\lambda}{2}z\overline{ s(z_0)}\right) b_k(z-s(z_0),z).
\end{align*}
The result hence follows from the formula for $b_k$ given in Theorem \ref{th:rappel}.
\end{proof}

Recall that for $k=\left(\begin{smallmatrix}A&B\\C&D\end{smallmatrix}\right)\in SU(p,q)$
we denote 
\begin{equation*}{\tilde k}=\begin{pmatrix}A&0&0&B\\
0&{\overline D}&{\overline C}&0\\
0&{\overline B}&{\overline A}&0\\
C&0&0&D\\
\end{pmatrix},\end{equation*} 
see Section \ref{sec:4}.

\begin{lemma} \label{lem:prepth} \begin{enumerate}
\item Let $k\in SU(p,q)$ and $z\in {\mathbb C}^n$. Let $z'=(k-I_n)z$. Then we have
\begin{equation*}\begin{pmatrix}s(z')\\ \overline{s(z')}\end{pmatrix}=({\tilde k}-I_{2n})
\begin{pmatrix}s(z)\\ \overline{s(z)}\end{pmatrix}.
\end{equation*}
\item For each $k\in SU(p,q)$, we have ${\tilde k}J{\tilde k}^t=J$ and ${\tilde k}^tJ{\tilde k}=J$. Moreover, if $\varphi$ is an analytic function on a domain of $M_{2n}({\mathbb C})$ which contains ${\tilde k}^t$ and ${\tilde k}^{-1}$ for some $k\in SU(p,q)$, then we have
$\varphi({\tilde k}^t)J=J\varphi({\tilde k}^{-1})$.
\end{enumerate}\end{lemma}

\begin{proof} (1) Easy computation. (2) For $k\in SU(p,q)$, we can verify that ${\tilde k}J{\tilde k}^t=J$ by a direct computation based on the relations on $A,B,C$ and $D$ given
at the beginning of Section \ref{sec:4}. The rest of the lemma follows immediately.
\end{proof}

\begin{theorem} \label{th:Wpi} Let $g=((z_0,{\bar z_0}),c_0,k)\in G$ such that $k+I_n$ is invertible. Then, for each $z\in 
{\mathbb C}^n$, we have 
\begin{align*}W&(\pi(g))(z)=2^ne^{i\lambda c_0}\vert \Det (k+I_n)\vert^{-1}
\exp \left(\tfrac{\lambda}{2}
\begin{pmatrix}z&{\overline z}\end{pmatrix}J({\tilde k}+I_{2n})^{-1}({\tilde k}-I_{2n})
\begin{pmatrix}z\\{\overline z}\end{pmatrix}\right)\\
\times &
\exp \left(\lambda 
\begin{pmatrix}z&{\overline z}\end{pmatrix}J({\tilde k}+I_{2n})^{-1}
\begin{pmatrix}s(z_0)\\ \overline {s(z_0)}\end{pmatrix}\right)\\
\times &
\exp \left(-\tfrac{\lambda}{4}
\begin{pmatrix}s(z_0)&\overline {s(z_0)}\end{pmatrix}J({\tilde k}+I_{2n})^{-1}
\begin{pmatrix}s(z_0)\\ \overline {s(z_0)}\end{pmatrix}\right).
\end{align*}\end{theorem}

\begin{proof} By using covariance of $W$ with respect to $\rho$ (see Section \ref{sec:3}),
we see that, for each $h\in H_n$, $k\in SU(p,q)$ and $z\in {\mathbb C}^n$ we have
\begin{equation*}W(\rho(h)^{-1}\sigma(k)\rho(h))(z)=W(\sigma(k))(h\cdot z)
\end{equation*} or, equivalently,
\begin{equation*}W(\rho(h^{-1}(k\cdot h))\sigma(k))(z)=W(\sigma(k))(h\cdot z).
\end{equation*}
We aim to use this relation to deduce a formula for $W(\pi(g))$ from the formula for
$W(\sigma(k))$ given by Theorem \ref{th:Wsigma}. 

Let $k\in SU(p,q)$ such that $k+I_n$ is invertible. Let $z_0\in  {\mathbb C}^n$. Assume,
moreover, that $k-I_n$ is also invertible and define $z_1:=(k-I_n)^{-1}z_0$ and
$h:=((z_1,{\bar z_1}),0)$. Then we have
\begin{align*}h^{-1}(k\cdot h)=&((-z_1,-{\bar z_1}),0)((kz_1,\overline {kz_1}),0)\\
=&((z_0,{\bar z_0}),-\tfrac{1}{2}\omega ((z_1,{\bar z_1}),(kz_1,\overline {kz_1})).
\end{align*}
Let us denote $c:=-\tfrac{1}{2}\omega ((z_1,{\bar z_1}),(kz_1,\overline {kz_1})$.
Then we have
\begin{equation*}W(\rho((z_0, {\bar z_0}),0)\sigma(k))(z)=
e^{-i\lambda c}W(\sigma(k))(h\cdot z).
\end{equation*}
By Theorem \ref{th:Wsigma}, we get
\begin{align*}W&(\rho((z_0,{\bar z_0}),0)\sigma(k))(z)=
2^n\vert \Det(k+I_{n})\vert^{-1}e^{-i\lambda c}\\
&\times 
\exp \left( \tfrac{\lambda}{2}\begin{pmatrix}z+s(z_1)&
\overline {z+s(z_1)}\end{pmatrix}J
({\tilde k}+I_{2n})^{-1}({\tilde k}-I_{2n})\begin{pmatrix}z+s(z_1)\\ \overline {z+s(z_1)}\end{pmatrix}\right).
\end{align*}
We are led to consider the following functions
\begin{align*}\varphi_1(z,k):=&\tfrac{\lambda}{2}
\begin{pmatrix}z&{\bar z}\end{pmatrix}J
({\tilde k}+I_{2n})^{-1}({\tilde k}-I_{2n})\begin{pmatrix}z\\{\bar z}\end{pmatrix};\\
\varphi_2(z,k):=&\tfrac{\lambda}{2}
\begin{pmatrix}s(z_1)&\overline {s(z_1)}\end{pmatrix}J
({\tilde k}+I_{2n})^{-1}({\tilde k}-I_{2n})\begin{pmatrix}z\\{\bar z}\end{pmatrix};\\
\varphi_3(z,k):=&\tfrac{\lambda}{2}
\begin{pmatrix}z&{\bar z}\end{pmatrix}J
({\tilde k}+I_{2n})^{-1}({\tilde k}-I_{2n})\begin{pmatrix}s(z_1)\\ \overline{ s(z_1)}\end{pmatrix};\\
\varphi_4(z,k):=&\tfrac{\lambda}{2}
\begin{pmatrix}s(z_1)&\overline {s(z_1)}\end{pmatrix}
J({\tilde k}+I_{2n})^{-1}({\tilde k}-I_{2n})\begin{pmatrix}s(z_1)\\ \overline{ s(z_1)}\end{pmatrix};\\
\varphi_5(z,k):=&-i\lambda c=\tfrac{\lambda}{2}i\omega ((z_1,{\bar z_1}),(kz_1,\overline {kz_1})).
\end{align*}
Note that
\begin{align*}\varphi_2(z,k)&=\tfrac{\lambda}{2}\begin{pmatrix}z&{\bar z}\end{pmatrix}({\tilde k}^t-I_{2n})({\tilde k}^t+I_{2n})^{-1}J^t\begin{pmatrix}s(z_1)\\ \overline{ s(z_1)}\end{pmatrix}\\
&=-\tfrac{\lambda}{2}\begin{pmatrix}z&{\bar z}\end{pmatrix}J({\tilde k}^{-1}-I_{2n})
({\tilde k}^{-1}+I_{2n})^{-1}\begin{pmatrix}s(z_1)\\ \overline{ s(z_1)}\end{pmatrix}\\
&=-\tfrac{\lambda}{2}\begin{pmatrix}z&{\bar z}\end{pmatrix}J(I_{2n}-{\tilde k})
(I_{2n}+{\tilde k})^{-1}\begin{pmatrix}s(z_1)\\ \overline{ s(z_1)}\end{pmatrix}
\end{align*}
by Lemma \ref{lem:prepth}. This implies that
\begin{align*}
\varphi_2(z,k)+\varphi_3(z,k)=&\lambda \begin{pmatrix}z&{\bar z}\end{pmatrix}
J({\tilde k}+I_{2n})^{-1}({\tilde k}-I_{2n})\begin{pmatrix}s(z_1)\\ \overline{ s(z_1)}\end{pmatrix}\\
=&\lambda \begin{pmatrix}z&{\bar z}\end{pmatrix}
J({\tilde k}+I_{2n})^{-1}\begin{pmatrix}s(z_0)\\ \overline{ s(z_0)}\end{pmatrix}
\end{align*}
by Lemma \ref{lem:prepth} again.

On the other hand, we have
\begin{align*}
\varphi_4(z,k)=&\tfrac{\lambda}{2}\begin{pmatrix}s(z_0)&\overline {s(z_0)}\end{pmatrix}
({\tilde k}^{t}-I_{2n})^{-1}J({\tilde k}+I_{2n})^{-1}\begin{pmatrix}s(z_0)\\ \overline{ s(z_0)}\end{pmatrix}\\
=&\tfrac{\lambda}{2}\begin{pmatrix}s(z_0)&\overline {s(z_0)}\end{pmatrix}J
{\tilde k}(I_{2n}-{\tilde k})^{-1}({\tilde k}+I_{2n})^{-1}\begin{pmatrix}s(z_0)\\ \overline{ s(z_0)}\end{pmatrix}
\end{align*} and also
\begin{align*}\varphi_5(z,k)=&\tfrac{\lambda}{2}
i\omega ((z_1,{\bar z_1}),(kz_1,\overline {kz_1})-(z_1,{\bar z_1}))\\
=&\tfrac{\lambda}{2}
i\omega ((z_1,{\bar z_1}),(z_0,{\bar z_0}))\\
=&-\tfrac{\lambda}{4}(z_1^+{\bar z_0}^+-z_1^-{\bar z_0}^--z_0^+{\bar z_1}^++
z_0^-{\bar z_1}^-)\\
=&\tfrac{\lambda}{4}\begin{pmatrix}s(z_0)&\overline {s(z_0)}\end{pmatrix}J
\begin{pmatrix}s(z_1)\\ \overline{ s(z_1)}\end{pmatrix}\\
=&\tfrac{\lambda}{4}\begin{pmatrix}s(z_0)&\overline {s(z_0)}\end{pmatrix}J({\tilde k}-I_{2n})^{-1}
\begin{pmatrix}s(z_0)\\ \overline{ s(z_0)}\end{pmatrix}.
\end{align*}
Hence we get
\begin{align*}\varphi_4(z,k)&+\varphi_5(z,k)\\
=&\tfrac{\lambda}{4}\begin{pmatrix}s(z_0)& \overline{ s(z_0)}\end{pmatrix}J
(I_{2n}-2{\tilde k}({\tilde k}+I_{2n})^{-1})({\tilde k}-I_{2n})^{-1}
\begin{pmatrix}s(z_0)\\ \overline{ s(z_0)}\end{pmatrix}\\
=&-\tfrac{\lambda}{4}\begin{pmatrix}s(z_0)& \overline{ s(z_0)}\end{pmatrix}J
({\tilde k}+I_{2n})^{-1}\begin{pmatrix}s(z_0)\\ \overline{ s(z_0)}\end{pmatrix}.
\end{align*}
By putting all these relations together and observing that
\begin{equation*}W(\rho((z_0,{\bar z_0}),0)\sigma(k))(z)=
2^n\vert \Det(k+I_{n})\vert^{-1}
\exp \left( \sum_{j=1}^{j=5}\varphi_j(z,k)\right),
\end{equation*}
we obtain the result.
\end{proof}

\noindent {\bf Example.} We take $p=q=1$. Let $k=\left(\begin{smallmatrix} e^{it}&0\\
0&e^{-it}\end{smallmatrix}\right)$ where $t\in {\mathbb R}$.
Let $g=((z_0,{\bar z_0}),c_0,k)$ where $z_0\in {\mathbb C}^2$, $c_0\in {\mathbb R}$.
Then we have ${\tilde k}=\left(\begin{smallmatrix} e^{it}I_2&0\\
0&e^{-it}I_2\end{smallmatrix}\right)$ and
\begin{equation*}({\tilde k}+I_4)^{-1}({\tilde k}-I_4)=
\begin{pmatrix} i\tan (\tfrac{1}{2}t)I_2&0\\
0&-i\tan (\tfrac{1}{2}t)I_2\end{pmatrix}.\end{equation*}
In this case, Theorem \ref{th:Wpi} gives
\begin{align*}W&(\pi(g))(z)=2(1+\cos(t))^{-1}e^{i\lambda c_0}\exp \left(
-i\lambda \vert z \vert^2\tan (\tfrac{1}{2}t)\right)\\
\times & \exp \left(\lambda (1+e^{it})^{-1}(e^{it}z\overline{ s(z_0)}-{
\bar z}s(z_0))\right) \exp \left(-i\tfrac{\lambda}{4}\vert z_0 \vert^2\tan (\tfrac{1}{2}t)
\right). \end{align*}

We can also give a formula for $W(d\pi(X))$, $X\in {\mathfrak g}$.

\begin{proposition} Let $Y=\left(\begin{smallmatrix} A&B\\
C&D\end{smallmatrix}\right)\in su(p,q)$, $z_0\in {\mathbb C}^n$, $c_0\in {\mathbb R}$.
Let $X=((z_0,{\bar z_0}),c_0, Y)\in {\mathfrak g}$. Then, for each $z=(z_1,z_2)\in {\mathbb C}^p \times {\mathbb C}^q$, we have
\begin{align*}S&(d\pi(X))(z)=\tfrac{\lambda}{2}(\overline{s(z_0)}z-{\bar z}s(z_0))+i\lambda c_0-\Tr(A)\\&
-\tfrac{\lambda}{2}\bigl(
({\overline C}\bar{z_1})\bar{z_2}+({\overline D}z_2)\bar{z_2}
+(A z_1)\bar{z_1}-z_1({\overline B}z_2)\bigr). \end{align*}
and
\begin{align*}W&(d\pi(X))(z)=\tfrac{\lambda}{2}(\overline{ s(z_0)}z-{\bar z}s(z_0))+i\lambda c_0\\&
-\tfrac{\lambda}{2}\bigl(
({\overline C}\bar{z_1})\bar{z_2}+({\overline D}z_2)\bar{z_2}
+(A z_1)\bar{z_1}-z_1({\overline B}z_2)\bigr). \end{align*}
\end{proposition}

\begin{proof} For  $X=((z_0,{\bar z_0}),c_0, Y)\in {\mathfrak g}$, we have
\begin{equation*}S(d\pi(X))(z)=S(d\rho((z_0,{\bar z_0}),c_0))(z)+S(d\sigma(Y))(z)\end{equation*} and we obtain the first equality by using Proposition \ref{prop:Wdsigma}.
The second equality can be proved by the same way.
\end{proof}
\vskip 0,5 cm
\textit {Acknowledgements.} I would like to thank the referee for many valuable comments and for pointing out some relevant references to me.
\vskip 0,5 cm

\end{document}